\documentclass[reqno,11pt]{amsart}
\usepackage{mathrsfs}
\usepackage{bbm}

\usepackage{amsmath, amssymb, amsthm, mathrsfs, graphicx}

\newtheorem{lemma}{Lemma}[section]
\newtheorem{theorem}{Theorem}[section]
\newtheorem{proposition}{Proposition}[section]
\newtheorem{corollary}{Corollary}[section]
\theoremstyle{definition}
\newtheorem{definition}{Definition}[section]
\theoremstyle{remark}
\newtheorem{remark}{Remark}[section]

\numberwithin{figure}{section}
\numberwithin{equation}{section}

\newcommand{\abs}[1]{\left\vert#1\right\vert}
\newcommand{\p}{\partial}

\newcommand{\norm}[1]{\left\Vert#1\right\Vert}
\newcommand{\dd}{\mathrm{d}}

\newcommand{\BV}{\mathrm{BV}}
\newcommand{\TV}{\mathrm{TV}}
\newcommand{\Lip}{\mathrm{Lip}\,}

\newcommand{\R}{{\mathbb R}}

\makeatletter
\newcommand{\rmnum}[1]{\romannumeral #1}
\newcommand{\Rmnum}[1]{\expandafter\@slowromancap\romannumeral#1@}
\makeatother

\begin{document}
\title[Well-Posedness of Transonic Characteristic Discontinuities]
{Well-Posedness of Transonic Characteristic Discontinuities in
Two-Dimensional Steady Compressible Euler Flows}

\author{Gui-Qiang Chen}
\author{Vaibhav Kukreja}
\author{Hairong Yuan}

\address{G.-Q. Chen, School of Mathematical Sciences, Fudan University, Shanghai 200433, China; Mathematical Institute, University of Oxford, Oxford, OX1 3LB, UK; Department of Mathematics, Northwestern University, Evanston, IL 60208, USA }
\email{chengq@maths.ox.ac.uk}

\address{Vaibhav Kukreja, Instituto de Matem$\acute{\text{a}}$tica Pura e Aplicada (IMPA),
 Rio de Janeiro, Brazil; Department of Mathematics, Northwestern University,
         Evanston, IL 60208, USA}
         \email{\tt vaibhav@impa.br; \tt vkukreja@math.northwestern.edu}
\address{H. Yuan,
Department of Mathematics, East China Normal University, Shanghai
200241, China}
\email{hryuan@math.ecnu.edu.cn;\ \
hairongyuan0110@gmail.com}

\keywords{$L^1$ stability, uniqueness, transonic, characteristic
discontinuity, vortex sheet, entropy wave, steady flow, Euler
system, Lagrange coordinates, front tracking}

\subjclass[2000]{35L50,\ 35L65,\ 35Q31,\ 35B35,\ 76H05,\ 76N10}

\date{\today}

\begin{abstract}
In our previous work,
we have established the existence of transonic characteristic discontinuities separating
supersonic flows from a static gas in two-dimensional steady
compressible Euler flows under a perturbation with small total
variation of the incoming supersonic flow over a solid right-wedge.
It is a free boundary problem in Eulerian coordinates and, across the
free boundary (characteristic discontinuity), the Euler equations are of
elliptic-hyperbolic composite-mixed type. In this paper, we further
prove that such a transonic characteristic discontinuity solution is
unique and  $L^1$--stable
with respect to the small perturbation of the incoming supersonic
flow in Lagrangian coordinates.
\end{abstract}

\maketitle

\section{Introduction}

We are concerned with the well-posedness of transonic
characteristic discontinuities that separate supersonic flows from a static
gas in two-dimensional steady compressible Euler flows. The
governing equations are the following Euler system that consists of
conservation laws of mass, momentum, and energy:
\begin{eqnarray}
&&\partial_x(\rho u)+\partial_y(\rho v)=0,\label{mass}\\
&&\partial_x(\rho u^2+p)+ \partial_y(\rho uv)=0, \label{momen1}\\
&&\partial_x(\rho uv)+\partial_y(\rho v^2+p)=0,\label{momen2}\\
&&\partial_x(\rho uE+pu)+\partial_y(\rho vE+pv)=0,\label{energy}
\end{eqnarray}
where $E=\frac{1}{2}(u^{2}+v^{2}) + e$.
The unknowns $\rho$, $p$, $e$, and
$(u,v)$ represent the density, pressure, internal energy, and
velocity of the fluid, respectively.
Specifically, for a polytropic
gas, the constitutive relations are
\begin{equation}\label{state}
p=\rho^\gamma \exp\big(\frac{S}{c_\nu}\big), \qquad e=\frac{1}{\gamma-1}\frac{p}{\rho},
\end{equation}
where $S$ is the entropy, $c_\nu$ is  a positive
constant, and $\gamma>1$ is the adiabatic exponent.
The sonic speed $c$
is determined by
\begin{equation}\label{c}
c=\sqrt{\frac{\p p}{\p \rho}}=\sqrt{\frac{\gamma p}{\rho}}.
\end{equation}
In our previous work \cite{Chen-Kukreja-Yuan-2012},
we have established the existence of a weak entropy solution
to the following initial--boundary value problem:
\begin{eqnarray}\label{probeuler}
\begin{cases}
\eqref{mass}-\eqref{energy},&\text{in} \ \ x>0, y>g(x),\\
U=U_0, &\text{on}\ \ x=0, y>0,\\
p=\overline{p},\ \  \ \frac{v}{u}=g'(x), &\text{on}\ \ x>0, y=g(x),
\end{cases}
\end{eqnarray}
provided that the incoming supersonic flow $U_0$ is close in $BV$
to a reference state $\overline{U}_+$, where $U=(u,v,p,\rho)$,
and $\overline{U}_+=(\overline{u}, 0, \overline{p},
\overline{\rho}_+)$ is a uniform supersonic flow; that is, it is a
constant vector and
$\overline{u}>\overline{c}_+=\sqrt{\frac{\gamma\overline{p}}{\overline{\rho}_+}}$.
The unknowns in problem \eqref{probeuler} are the supersonic flow $U=U(x,y)$
and the free
boundary $\mathcal{D}_U=\{y=g(x), x\ge0\}$, which is a Lipschitz
curve passing through the origin. The free boundary is actually a transonic
characteristic discontinuity (vortex sheet and/or entropy wave)
separating the supersonic flow $U$ from the unperturbed static gas
$\overline{U}_-=(0,0,\overline{p}, \overline{\rho}_-)$ that is
subsonic (cf. Figure \ref{fig1}). Thus, problem \eqref{probeuler} is a free boundary problem
(in the Euler coordinates) and, across the free boundary, the Euler
equations are of elliptic-hyperbolic composite-mixed type.

\begin{figure}[h]\label{fig1}
\centering
  \setlength{\unitlength}{1bp}%
  \begin{picture}(300, 200)(10,0)
  \put(0,0){\includegraphics[scale=0.55]{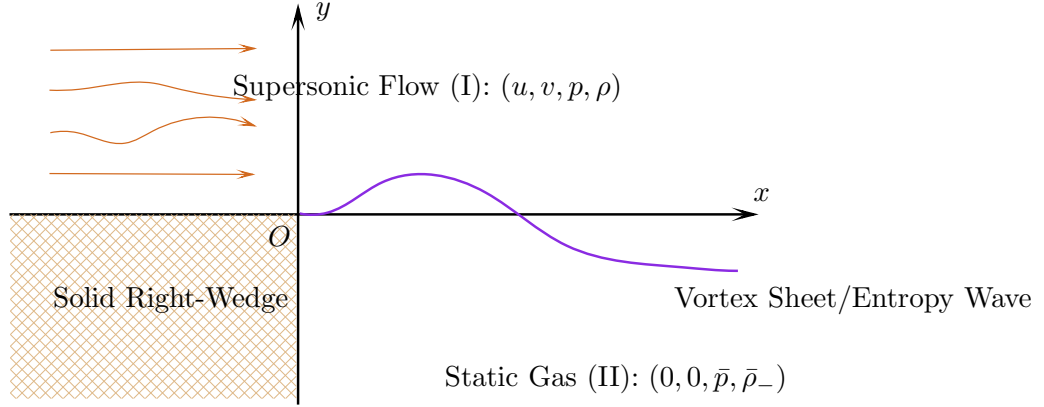}}

   \put(280,80){ ${ x}$}
     \put(115,150){ ${ y}$}
     \put(97,63){ ${ O}$}
\put(16,40){ Solid Right-Wedge}
\put(80,120){ { Supersonic Flow
(\Rmnum{1}): $(u,v,p,\rho)$}}
\put(160,10){ { Static Gas (\Rmnum{2}):
$(0,0,\bar{p},\bar{\rho}_-)$}}
\put(250,40){ {Vortex Sheet/Entropy Wave}}
\end{picture}
\caption{For supersonic flow passing the corner $O$ with a static gas $\overline{U}_-$
on the right of the solid right wedge (i.e., with velocity zero),
then a combined vortex sheet/entropy wave is generated to separate the static gas  below from the supersonic
flow above. }
\end{figure}

To further study the $L^1$ stability, we run into a
difficulty to define the $L^1$--distance between two solutions,
since two solutions $U$ and $V$ describing the supersonic flows are generically
defined in different domains depending on the location of their
respective free boundaries $\mathcal{D}_U$ and $\mathcal{D}_V$:
For some given ``time'' $x^{\ast}> 0$ and at some point $y^{\ast} \in \mathbb{R}$,
it may happen that, at
$(x^\ast, y^\ast)$, it is a static state for one solution, but is a supersonic flow
for the other.
To this end, we employ the special structure
of the Euler equations of the two-dimensional steady flows, and
introduce the Lagrangian coordinates $(\xi, \eta)$ associated to a solution, as done
in the study of transonic shocks
\cite{SChen-2006,Fang-Wang-Yuan-2011,Yuan-2006}, to transform the free
boundary to the positive $\xi$-axis, so that the solutions $U$ and $V$
are then defined in the same domain $\{\xi>0, \eta>0\}$ in
Lagrangian coordinates.
Then the initial data functions
$U_0(0,y)$ and $V_0(0,y)$ are transformed to the corresponding initial data
functions $U_0(0,\eta)$ and $V_0(0,\eta)$
in their respective Lagrange
coordinates.
Similar to the analysis in
\cite{Chen-Kukreja-Yuan-2012} for the construction of approximate
solutions, we construct the approximate front tracking solutions
$U(\xi,\eta)$ and $V(\xi,\eta)$ in the domain $\{\xi>0, \eta>0\}$
with the initial data $U_0(0,\eta)$ and $V_0(0,\eta)$ on $\xi=0$,
and the boundary data $p=\overline{p}$ on $\eta=0$ respectively, and
study the $L^1$--stability property of these solutions for the
initial-boundary value problem \eqref{prob} in Lagrangian coordinates,
as proposed accurately below.  With this merit of Lagrangian
coordinates, we can show the characteristic discontinuity solutions
are both unique for the given incoming supersonic flow in the Euler coordinates
and $L^1$--stable with respect to the small perturbations of the incoming supersonic
flows in the Lagrangian coordinates.

In Section 2, we present the Lagrangian
coordinates associated to a solution to problem \eqref{probeuler} and
formulate this problem in the Lagrangian coordinates as problem
\eqref{prob}. Then we briefly review how the existence of a weak
entropy solution is shown by using the front tracking method.
Section 3 is devoted to establishing the $L^1$--stability of problem
\eqref{prob}. Then, in Section 4, we show the existence of a semigroup associated
with problem \eqref{prob}. Finally, in Section 5, we show the
uniqueness of solutions to problem \eqref{prob} in the larger class of
viscosity solutions. By Wagner's Theorem \cite[Theorem 2]{Wagner-1987},
there is a one-to-one correspondence between
a bounded measurable weak solution of a hyperbolic system of
conservation laws and a bounded measurable weak solution of the
corresponding equations in Lagrangian coordinates. Thus, the uniqueness
results established for problem \eqref{prob} imply the uniqueness of
transonic characteristic discontinuities in Eulerian coordinates. This is
clarified in Section 6. The well-posedness results are summarized in
Theorem \ref{thmmain}. These results on the existence, uniqueness, and
$L^1$--stability in Lagrangian coordinates,
together with the existence theorem established in
\cite{Chen-Kukreja-Yuan-2012}, may be considered as a complete
mathematical well-posedness theory of such transonic characteristic
discontinuity solutions.

\section{Formulation of the Problem in Lagrangian Coordinates}

In this section we reformulate the problem in Eulerian coordinates to
the problem in Lagrangian coordinates and show the existence of entropy
solutions for the problem in Lagrangian coordinates.

\subsection{Lagrangian coordinates}
For a
piecewise $C^1$ smooth flow,
we may introduce the
Lagrangian transformation as used in \cite{Chen-Fang-2008,Yuan-2006} to
formulate problem \eqref{probeuler} in Lagrangian
coordinates, which enables us to straighten the streamlines and
hence treat a strict hyperbolic system derived from
\eqref{mass}--\eqref{energy}.

Let $U\in L^\infty(\mathbb{R}^+)$ and $g\in \Lip(\mathbb{R}^+)$ be
a {\it piecewise $C^1$} weak
entropy solution to problem \eqref{probeuler}.
Define
\begin{eqnarray}
\eta=\eta(x,y;x_0,y_0)=\int_{(x_0, y_0)}^{(x,y)} \rho u(s,t)\,\dd
t-\rho v(s,t)\,\dd s,
\end{eqnarray}
where $(x_0,y_0)$ is a fixed point on the transonic characteristic
discontinuity $\mathcal{D}_U$, and the integration is on any smooth
curve $\Gamma$ connecting $(x_0,y_0)$ with $(x,y)$ and lies in the
upper side of $\mathcal{D}_U$.
Since the domain $\{x>0, y>g(x)\}$ is
simply connected, by the conservation of
mass,
as well as the Rankine-Hugoniot (R-H) conditions for crossing the shock-front,
$\eta$ is a well-defined function of $(x,y)$ with $y\ge g(x)$, independent of
the choice of $\Gamma$. Clearly, we have
\begin{eqnarray}\label{etaeq}
\frac{\p \eta}{\p x}=-\rho v,\quad \frac{\p \eta}{\p y}=\rho u.
\end{eqnarray}
We also note that, by the last condition in \eqref{probeuler},
$\eta$
is independent of $(x_0,y_0)$ on $\mathcal{D}_U$.
In fact, for $(x_0,y_0)$  and $(x_0',
y_0')$ on the characteristic discontinuity, we have
\begin{eqnarray}
\eta(x,y; x_0,y_0)-\eta(x,y;
x_0',y_0')&=&\int_{(x_0,y_0)}^{(x_0',y_0')}\rho u(s,t)\, \dd t-\rho
v(s,t)\, \dd s\nonumber\\
&=&\int_{x_0}^{x_0'}(\rho u(s,g(s))\cdot g'(s)-\rho
v(s,g(s)))\,\dd s\nonumber\\
&=&0.
\end{eqnarray}
Since the characteristic discontinuity $g(s)$ is in $\Lip(\mathbb{R}^+,
\mathbb{R})$, it is differentiable almost everywhere,
so that the integrand
is discontinuous only at these points in a set of Lebesgue measure
zero, which is harmless for the above calculation. Hence, in the
following, we may write $\eta=\eta(x,y)$ with $\eta=0$ whenever
$y=g(x)$.

If $u, v$, and $\rho$ belong to $L^\infty$, and \eqref{mass} is satisfied
in the sense of distributions, we can also find a unique Lipschitz
continuous function $\eta$ so that $\eta(x_0,y_0)=0$ and
\eqref{etaeq} holds. As a matter of fact, using standard
mollification, we may approximate $\rho u$ and $\rho v$ in weak
convergence of distributions by $C^\infty$ functions $(\rho
u)^\epsilon$ and $(\rho v)^\epsilon$, for which the equality
$\p_x(\rho
u)^\epsilon+\p_y(\rho v)^\epsilon=0$ still holds.
Furthermore, by the Young
inequality, we have $\norm{(\rho u)^\epsilon}_{L^\infty}\le
C\norm{\rho u}_{L^\infty}$ and $\norm{(\rho
v)^\epsilon}_{L^\infty}\le C\norm{\rho v}_{L^\infty}$.
Then, as above, we can solve $\eta^\epsilon\in C^1$ so that
$\eta^\epsilon(x_0,y_0)=0$ and hence, in any bounded domain,
$\{\eta^\epsilon\}$ is a family of $C^1$ functions that are
uniformly bounded and equicontinuous. Thus, there is a subsequence
$\{\eta^k\}$ that converges uniformly to some $\eta$; by taking a
diagonal subsequence, we can find $\eta$ that is defined in the
whole domain and $\eta^k$ converges uniformly to $\eta$ in any
compact subregion, hence $\eta^k\to \eta$ in $\mathscr{D}'$. It is
obvious that $\eta(x_0,y_0)=0$, and \eqref{etaeq} holds by uniqueness
of limits in the sense of distributions. This then ensures that
$\eta$ is Lipschitz continuous. The uniqueness follows from the
well-known fact that a distribution with zero derivatives must be a
constant. Differentiating $\eta$ along a streamline (which is
Lipschitz continuous) and using \eqref{etaeq} yield that
it is constant
along the stream line.

Now we introduce the following Lagrangian transformation
$(x,y)\mapsto (\xi,\eta)$:
\begin{equation}\label{eq:3.10}
(\xi, \eta)=(x, \eta(x, y)).
\end{equation}
Then we have
\begin{equation}\label{eq:3.12}
\frac{\p(\xi,\eta)}{\p(x,y)}=\left(%
\begin{array}{cc}
  1 & 0 \\
  -\rho v & \rho u \\
\end{array}%
\right),
\end{equation}
thus
\begin{equation}\label{eq:3.13}
\p_x=\p_\xi-\rho v\p_\eta, \quad
\p_y=\rho u\p_\eta.
\end{equation}
This transform is Lipschitz continuous and one-to-one provided
$\rho u>0$. A simple computation shows that equations
\eqref{mass}--\eqref{momen2} may be written in divergence form:
\begin{equation}\label{Lagr}
\begin{cases}
\displaystyle{\p_\xi\Big(\frac{1}{\rho u}\Big)-\p_\eta\big(\frac{v}{u}\big)=0},\
\\
\p_\xi(u+\frac{p}{\rho u})-\p_\eta \big(\frac{pv}{u}\big)=0,
\ \\
\p_\xi v+\p_\eta p=0, \
\end{cases}
\end{equation}
or, as a symmetric system for $U=(u,v,p)^{\top}$,
\begin{equation}\label{eq:3.15}
A\p_\xi U+\, B\p_\eta U=0,
\end{equation}
with
\begin{equation}\label{eq:3.16}
{A}=\left(%
\begin{array}{ccc}
  u & 0 & \frac 1 \rho \\
  0 & u & 0 \\
  \frac 1 \rho & 0 & \frac{u}{\rho^2c^2} \\
\end{array}%
\right),\qquad B=\left(%
\begin{array}{ccc}
  0 & 0 & -v \\
  0 & 0 & u \\
  -v & u & 0 \\
\end{array}%
\right).
\end{equation}
For $\rho u\ne0$, the conservation law of energy becomes
$\p_\xi(\frac{u^2+v^2}{2}+\frac{c^2}{\gamma-1})=0$, that is,
\begin{eqnarray}\label{Bernoullilaw}
\frac{u^2+v^2}{2}+\frac{c^2}{\gamma-1}=b(\eta).
\end{eqnarray}
This is the well-known Bernoulli law. As $b(\eta)$ is given by the
initial data, in the following we focus on system \eqref{Lagr}
with $\rho$ determined by $U=(u,v,p)$ through \eqref{Bernoullilaw}.

The eigenvalues $\lambda$ of \eqref{eq:3.15} (i.e.
$|\lambda {A}- {B}|=0$) are
\begin{eqnarray}
{\lambda}_1&=&\frac{\rho
c^2u}{u^2-c^2}\left(\frac{v}{u}-\sqrt{M^2-1}\right),\label{lambda1new}\\
{\lambda}_2&=&0,\label{lambda2new}\\
{\lambda}_3&=&\frac{\rho
c^2u}{u^2-c^2}\left(\frac{v}{u}+\sqrt{M^2-1}\right),\label{lambda3new}
\end{eqnarray}
where $M=\frac{\sqrt{u^2+v^2}}{c}$ is the Mach number of the flow.
Then, for $u>c$, system \eqref{eq:3.15} is strictly
hyperbolic. The associated right-eigenvectors are
\begin{eqnarray}
r_1&=&\kappa_1(\frac{\lambda_1}{\rho}+v, -u, -\lambda_1u)^\top,\\
r_2&=&(u,v,0)^\top,\\
r_3&=&\kappa_3(\frac{\lambda_3}{\rho}+v, -u, -\lambda_3u)^\top,
\end{eqnarray}
where $\kappa_j$ can be chosen that ${r}_j \cdot \nabla\lambda_j
\equiv 1$, since the $j$th-characteristics
fields, $j=1,3$, are genuinely nonlinear.
Note that the second
characteristic field is always linearly degenerate: ${r}_2 \cdot
\nabla \lambda_2 =0$.

\subsection{Formulation of the problem in Lagrangian coordinates}
As we noted that the transonic characteristic discontinuity becomes the
positive $\xi$-axis in Lagrangian coordinates, we formulate the
problem in Eulerian coordinates into the following initial-boundary
value problem for equations \eqref{Lagr}:
\begin{eqnarray}\label{prob}
\begin{cases}
\eqref{Lagr}&\text{in}\ \ \xi>0, \eta>0,\\
U(0,\eta)=U_0(\eta) &\text{on}\ \ \xi=0, \eta>0,\\
p=\overline{p} &\text{on}\ \ \xi>0, \eta=0.
\end{cases}
\end{eqnarray}
Once we solved $U$ from this problem, we then obtain the free boundary
in Eulerian coordinates
\begin{eqnarray}\label{eqfree}
g(x)=\int_0^x\frac{v}{u}(\xi,0)\,\dd \xi.
\end{eqnarray}
These are the corresponding forms in Lagrangian coordinates
for problem \eqref{probeuler}.

\subsection{Existence of entropy solutions}
The entropy solutions of problem \eqref{prob} can be defined in the
standard way via integration by parts.
We note that the existence of entropy solutions to problem \eqref{prob} can be constructed easily by using the front tracking method as carried out in \cite{Chen-Kukreja-Yuan-2012}
provided that
$\norm{U_0-\overline{U}_+}_{\BV(\mathbb{R}^+)}$ is sufficiently
small. In particular, 
the following lateral Riemann problem with the boundary
data $p=\overline{p}$ on the
characteristic boundary $\{\eta=0\}$ is uniquely solvable.

\begin{lemma}\label{Lemma 2.1}
Consider the following lateral Riemann problem:
$$
\begin{cases}
\eqref{Lagr} & \text{in}\ \ \xi>0,\eta>0,\\
U=U_+ &\text{on}\ \ \xi=0,\eta>0,\\
p=\overline{p} &\text{on}\ \ \xi>0,\eta=0.
\end{cases}
$$
There exists $\varepsilon>0$ so that, if $U_+$ lies in the
ball $O_\varepsilon(\overline{U}_+)$ with center $\overline{U}_+$
and radius $\varepsilon$, then there is a unique admissible
solution that contains only a $3$-wave.
\end{lemma}

\begin{proof}
1. Note that system \eqref{Lagr} is strictly hyperbolic for $u>c$.
For each point $U$ with $u>c$, in its small neighborhood, we can
obtain $C^2$-wave curves $\Phi_j(\alpha; {U}), j=1,2,3$, so that
$\Phi_j(0;{U})={U}$ and $\frac{\dd\Phi_j}{\dd
\alpha}|_{\alpha=0}=r_j({U})$: $\Phi_j(\alpha; {U})$ is connected to
${U}$ from the upper side by a simple wave of $j$-family with strength
$|\alpha|$; for $\alpha>0$ and $j=1,3$, this wave is a rarefaction
wave, while, for $\alpha<0$ and $j=1,3$, this wave is a shock. For
$j=2$, the wave is always a characteristic discontinuity.

For our purpose, we note that there is also a $C^2$-curve $\Psi_3(\beta;{U})$
which consists of those states that can be connected to ${U}$ from
the lower side by a 3-wave of strength $\beta$. We have
$\Psi_3(0;{U})={U}$ and $\frac{\dd\Phi_j}{\dd
\beta}|_{\beta=0}=-r_3({U})$.

2. We set $U=(u,v,p)^\top$ in Lagrangian coordinates and use
$U_3$ to represent the third argument of the vector $U$. Then, to
solve the lateral Riemann problem, it suffices to show that there exists
a unique $\beta$ so that
$(\Psi_3(\beta;{U}_+))_3=\overline{p}$. Therefore,
we consider the following function:
$$
L(\beta;U_+)=\Psi_3(\beta;{U}_+))_3-\Psi_3(0;\overline{U}_+))_3.
$$
It is clear that $L(0;\overline{U}_+)=0$, and
$$
\frac{\p L(0;\overline{U}_+)}{\p\beta}
=-r_3(\overline{U}_+)_3=\kappa_3\bar{\lambda}_3 u\ne0.
$$
By the implicit function theorem, there exists $\varepsilon>0$ such that,
for $U_+\in O_\varepsilon(\overline{U}_+)$, there is a function
$\beta=\beta(U_+)$ so that $L(\beta(U_+);U_+)=0$. Then, similar to
Lemma 2.2 in \cite{Chen-Kukreja-Yuan-2012}, by using the Taylor
expansion up to second order (recall that $\Psi_j$ is $C^2$),
we obtain the following estimate:
\begin{eqnarray}
\beta=K|p_+-\overline{p}|+O(1)|U_+-\overline{U}_+|^2,
\end{eqnarray}
with a suitable constant $K$ depending only on $\overline{U}_+.$
\end{proof}

For the reflection of weak 1-waves off the characteristic boundary
$\{\eta=0\}$, we have

\begin{lemma}\label{lem:reflection}
Suppose that $U^l,U^m$, and $U^r$ are three states in
$O_\epsilon(\overline{U}^+)$ for sufficiently small $\epsilon$,
with $U^m=\Phi_1(\alpha_1;U^l)=\Phi_3(\alpha_3;U^r)$. Then
\begin{eqnarray}\label{reflection}
\alpha_3=-K_2\alpha_1+M_2|\alpha_1|^2,
\end{eqnarray}
with the constant $K_2>0$ and the quantity $M_2$ bounded in
$O_\epsilon(\underline{U}^+)$.
Furthermore, for $U^l = (u_l, v_l, p_l)$,
$|K_2| > 1, |K_2| < 1, \text{ and } |K_2| =1 $ when $v_l<0$, $v_l>0$, \text{ and } $v_l=0$,
respectively.
\end{lemma}

The proof here follows the similar argument for Lemma 2.4 in \cite{Chen-Kukreja-Yuan-2012}.

For our current problem, we have to modify slightly the Glimm functional as done in \cite{Chen-Kukreja-Yuan-2012}. In particular, this will allow us to show later that the Lyapunov functional $\Phi$ (see Section 3) decreases when a weak wave of 1-family is reflected off the boundary. We now introduce the following version of the Glimm functional:
\begin{eqnarray*}
\mathcal{G}(\xi)=\mathcal{V}(\xi)+\kappa \mathcal{Q}(\xi),
\end{eqnarray*}
where $\kappa > 0$ is a large constant to be chosen. The terms $\mathcal{V}$
and $\mathcal{Q}$ are explained below:

\medskip
{$\bullet$\it The weighted strength term $\mathcal{V}(\xi)$.} We define the total (weighted) strength
of weak waves in $U^{\delta}(\xi, \cdot)$ as
\begin{eqnarray}
\mathcal{V}(\xi) = \sum_{\alpha} |b_{\alpha}|.
\end{eqnarray}
Here, for a weak wave $\alpha$ of $i_\alpha$-family, we
define its weighted strength as
\begin{equation*}
b_{\alpha} =
\begin{cases} k_{+}\alpha & \text{if $i_{\alpha} =1$},
\\
\alpha &\text{if $i_{\alpha} =2,3$},
\end{cases}
\end{equation*}
where $k_{+} > K_{2}$ is a positive constant which has been fixed by considering the reflection coefficient $K_{2}$ in the strength $\alpha_{3}$ of the reflected 3-waves in the interaction between weak 1-wave and the characteristic boundary $\{\eta=0\}$ as given in Lemma \ref{lem:reflection}.

\medskip
{$\bullet$ \it The interaction potential term $\mathcal{Q}(\xi)$.} The
interaction potential term we used here is a modification to the one
introduced by Glimm, that is,
\begin{eqnarray}
\mathcal{Q}(\xi)= \sum_{(\alpha, \beta) \in \mathscr{A}} |b_{\alpha} b_{\beta}| + \sum_{\beta \in \mathscr{A}_{b}} |b_{\beta}|
= \mathcal{Q}_{\mathcal{A}} + \mathcal{Q}_{b},
\end{eqnarray}
where the set $\mathscr{A}(\xi)$ is of all the couples $(b_{\alpha},b_{\beta})$
of approaching wave-fronts.
The term $\mathcal{Q}_{b}$ in our wave interaction potential is an additional term, in comparison with the Cauchy problem; and if a weak wave $\alpha$ of $1$-family is approaching the boundary, we write $\alpha \in \mathscr{A}_{b}$.

Now, for given initial data $U_0$, we may construct an approximate
solution $U^\delta$ by a front tracking algorithm introduced by
\cite{Holden-Risebro-2002}, where $\delta$ is a small parameter
measuring the accuracy of the solution, which controls the following
four types of errors generated by the algorithm:
\begin{itemize}
\item Errors in the approximation of initial data;
\item Errors in the speeds of shock, characteristic discontinuities (vortex sheet and entropy wave), and rarefaction fronts;
\item Errors by approximating the rarefaction waves by piecewise constant rarefaction fronts;

\item Errors from removing all the fronts with generation higher
than $N \in \mathbb{Z}_+$ ($N$ depends on $\delta$).
\end{itemize}

The construction of a Glimm functional as above
provides the necessary uniform
estimates that guarantee the existence of a subsequence of  $U^\delta$
which converges to a bounded entropy solution of \eqref{prob} in
$C([0,T]; L^1(\mathbb{R}^+))$ for any $T>0$.

The following sections are devoted to the $L^1$--stability and
uniqueness properties for problem \eqref{prob}.

\section{A Lyapunov functional and $L^1$ stability}

This section is devoted to establishing the $L^1$--stability of two
entropy solutions ${U}$ and ${V}$ of problem \eqref{prob} with respective initial
data $U_0$ and $V_0$ obtained
by the front tracking method:
\begin{eqnarray}\label{eq31}
\norm{U(\xi,\cdot)-V(\xi,\cdot)}_{L^1(\mathbb{R}^+)}\le
C \norm{U_0-V_0}_{L^1(\mathbb{R}^+)},
\end{eqnarray}
with a constant $C$ depending only on the equations and the
reference state $\overline{U}_+$. To this end, following
an idea in \cite{Bressan-Liu-Yang-1999}, we
introduce a Lyapunov functional $\Phi(U^{\delta_1},V^{\delta_2})$ by
incorporating additional new waves generated from the weak wave
interactions with the characteristic boundary $\{\eta=0\}$,
which is equivalent to the
$L^{1}(\mathbb{R}^+)$--distance:
\begin{eqnarray}\label{l1equiv}
C^{-1}_{1} \norm{U(\xi,\cdot) - V(\xi,\cdot)}_{L^{1}} \leq
\mathit{\Phi}(U, V) \leq C_{1}\|U(\xi,\cdot) -
V(\xi,\cdot)\|_{L^{1}}
\end{eqnarray}
and decreases as $\xi$ increases:
\begin{eqnarray}\label{decreases}
\mathit{\Phi}(U^{\delta_1}(\xi_{2},\cdot),
V^{\delta_2}(\xi_{2},\cdot)) -
\mathit{\Phi}(U^{\delta_1}(\xi_{1},\cdot),
V^{\delta_2}(\xi_{1},\cdot))
\leq
C_{2}\max(\delta_1,\delta_2)(\xi_{2} - \xi_{1}), \hspace{2mm}
\forall \hspace{2mm}
 \xi_{2} > \xi_1 > 0,
\end{eqnarray}
for some constants $C_i$, $i$ = 1, 2, where $U^{\delta_1}$ and
$V^{\delta_2}$ are two approximate solutions corresponding to
the initial data $U_0$ and $V_0$ obtained by the wave-front tracking,
with accuracy $\delta_1$ and $\delta_2$, respectively.

Estimate \eqref{decreases} has many important implications.
Firstly, taking $U_0=V_0$ and $\xi_1=0$ in \eqref{decreases} and
using \eqref{l1equiv}, we obtain
$\lim_{\delta_1, \delta_2\to 0}
\norm{U^{\delta_1}(\xi,\cdot)-U^{\delta_2}(\xi,\cdot)}_{L^1(\mathbb{R}^+)}=0$.
Thus, the whole sequence of approximate solutions $\{U^\delta\}$
converges to the same limit. Secondly, taking directly
$\delta_1,\delta_2\to 0$ yields \eqref{eq31}.
This again implies that the entropy solution of problem \eqref{prob}
obtained by the front tracking method is unique.

In the following, we first define specifically the Lyapunov
functional and then verify \eqref{l1equiv} and \eqref{decreases}.

\subsection{Definition of the Lyapunov functional and its equivalence
to the $L^1$--distance}
Similar to \cite{Bressan-Liu-Yang-1999, Lewicka-Trivisa-2002,
Liu-Yang-1999}, when $\xi$ is fixed, for each $\eta\ge0$, we solve a
``Riemann problem" with the below-state $U(\xi,\eta)$, the upper-state
$V(\xi,\eta)$, by moving along the Hugoniot curves $S_{1},
C_{2}$, and $S_{3}$ of system \eqref{Lagr} in the phase space.
We use $h_{i}(\xi,\eta)$ to denote the strength of the $i$-Hugoniot
wave in this ``Riemann solver", which is totally determined by
$U(\xi,\eta)$ and $V(\xi,\eta)$. As in the standard Riemann
solver, $\sum_{i=1}^3|h_i(\xi,\eta)|$ is equivalent to
$|U(\xi,\eta)-V(\xi,\eta)|$, with a constant depending only on
$\overline{U}_+$.

Now we define the weighted $L^{1}$--strengths:
\begin{eqnarray}
q_{i}(\xi,\eta) = c^{a}_{i}h_{i}(\xi,\eta),
\end{eqnarray}
where the constants $c^{a}_{i}$ are to be determined later, depending
on the estimates on the wave interactions and reflections off the characteristic
boundary $\{\eta=0\}$ (as done in Section 2).

Then we define the Lyapunov functional:
\begin{equation}
\mathit{\Phi}(U(\xi,\cdot), V(\xi,\cdot)) = \sum_{i =
1}^{3}\int\limits_{0}^{\infty}\
|q_{i}(\xi,\eta)|W_{i}(\xi,\eta)\,\dd\eta,
\end{equation}
with the weight
\begin{equation}\label{3.6}
W_{i}(\xi,\eta) = 1 + \kappa_{1}\mathcal{A}_{i}(\xi,\eta) +
\kappa_{2}\big(\mathcal{Q}(U)(\xi) + \mathcal{Q}(V)(\xi)\big).
\end{equation}
Here $\kappa_{1}$ and $\kappa_{2}$ are two constants to be specified
later; $\mathcal{Q}(U)(\xi)$ and $\mathcal{Q}(V)(\xi)$
are the Glimm's total wave interaction potentials in
$U$ and $V$ at ``time" $\xi$, respectively;
$\mathcal{A}_{i}(\xi,\eta)$ is the total strength of
waves in $U$ and $V$ which approach the $i$-wave $q_{i}(\xi,\eta)$
at ``time" $\xi$,  defined by
\begin{equation}
\mathcal{A}_{i}(\xi,\eta) = \mathcal{B}_{i}(\xi,\eta) + \mathcal{C}_{i}(\xi,\eta)
\end{equation}
and
\begin{eqnarray}
\mathcal{B}_{i} &=& \Biggl(\sum_{\substack{\alpha \in \mathcal{J}(U) \cup
\mathcal{J}(V) \\ \eta_{\alpha} < \eta, i < k_{\alpha} \leq 3}} +
\sum_{\substack{\alpha \in \mathcal{J}(U) \cup \mathcal{J}(V) \\
\eta_{\alpha} > \eta, 1 \leq k_{\alpha} < i}}\Biggr) |\alpha |,  \nonumber \\
\mathcal{C}_{i} &=& \left\{ \begin{array}{ll}
(\sum_{\substack{\alpha \in \mathcal{J}(U), \eta_{\alpha} < \eta, k_{\alpha}
= i}} + \sum_{\substack{\alpha \in \mathcal{J}(V), \eta_{\alpha} > \eta,
k_{\alpha} = i}})|\alpha| &\quad \mbox{ if $q_{i}(\eta) < 0$,} \\
(\sum_{\substack{\alpha \in \mathcal{J}(V), \eta_{\alpha} < \eta,
k_{\alpha} = i}} + \sum_{\substack{\alpha \in \mathcal{J}(U),
\eta_{\alpha} > \eta, k_{\alpha} = i}})|\alpha| &\quad \mbox{ if
$q_{i}(\eta) > 0$,}
        \end{array} \right.       \nonumber
\end{eqnarray}
where $\mathcal{J}(U)$ and $\mathcal{J}(V)$ are the sets of fronts at ``time" $\xi$
in $U$ and $V$, respectively,
$\eta_\alpha$ is the position of the front $\alpha$, and $k_\alpha$
is the characteristic family associated to the front $\alpha$.
In the following, we sometimes drop the dependence of $\xi$ and
write simply $q_i(\eta),
W_i(\eta), \mathcal{A}_i(\eta)$,  etc. when no confusion arises.

Since, for any $U(\xi, \cdot)$, $V(\xi, \cdot)$ $\in \BV $
with $TV\big({U}(\cdot)\big) + TV\big({V}(\cdot)\big)$ sufficiently small
(depending only on $\kappa_1, \kappa_2$, and $c_i^a$), we have
\begin{align}
& C^{-1}_{0} \norm{U(\xi,\cdot) -
V(\xi,\cdot)}_{L^{1}(\mathbb{R}^+)} \leq
\sum_{i=1}^{3}\int\limits^{\infty}\limits_{0}
|q_{i}(\eta)|\,d\eta
\leq C_{0}\|U(\xi,\cdot) - V(\xi,\cdot)\|_{L^{1}(\mathbb{R}^+)}, \nonumber \\
& 1 \leq W_{i}(\eta) \leq 2, \qquad i = 1, 2, 3, & \nonumber
\end{align}
for some constant $C_{0}$ depending essentially only on the
system and $\overline{U}_+$. Therefore, for any $\xi \ge 0$,
\begin{equation}\label{eqlnorm}
C^{-1}_{1} \norm{U(\xi,\cdot) - V(\xi,\cdot)}_{L^{1}(\mathbb{R}^+)}
\leq \mathit{\Phi}(U, V) \leq C_{1}\|U(\xi,\cdot) -
V(\xi,\cdot)\|_{L^{1}(\mathbb{R}^+)},
\end{equation}
where $C_{1}$ is independent of $\xi$, $U$, and $V$.

\subsection{Decreasing of the Lyapunov functional} From now on, we
examine how the Lyapunov functional
$\mathit{\Phi}(U^{\delta_1}(\xi,\cdot), V^{\delta_2}(\xi,\cdot))$
evolves in the flow direction $\xi > 0$, where $U^{\delta_1}$ and
$V^{\delta_2}$ are approximate solutions obtained by the front tracking
method from the initial data $U_0$ and $V_0$, respectively,
and we also
set $\delta=\max(\delta_1,\delta_2)$ to control the errors. For
simplicity, in the following, we also write $U^{\delta_1}$ and
$V^{\delta_2}$ as $U$ and $V$, respectively.

Denote $\lambda_{i}(\eta)$ the speed of the $i$-wave $q_{i}(\eta)$ along
the Hugoniot curve in the phase space (it is not the characteristic speed
except for the characteristic discontinuity).
At $\xi> 0$ which is not an interaction ``time" of the waves either in $U$
or $V$, for $\mathcal{J}=\mathcal{J}(U)\cup\mathcal{J}(V)$ (it is a
finite set), it holds that
\begin{align}
& {\dd \over \dd\xi}\mathit{\Phi}\left(U(\xi), V(\xi)\right) \nonumber \\
&= \sum_{\alpha \in \mathcal{J}}\sum_{i = 1}^{3}\Big(
\abs{q_{i}(\eta_{\alpha}^{-})}W_{i}(\eta_{\alpha}^{-}) -
\abs{q_{i}(\eta_{\alpha}^{+})}W_{i}(\eta_{\alpha}^{+})\Big)\dot{\eta}_{\alpha} + \sum_{i = 1}^{3}\abs{q_{i}(b)}W_{i}(b)\dot{\eta}_{b}    \nonumber \\
&= \sum_{\alpha \in \mathcal{J}}\sum_{i = 1}^{3}\Big(\abs{q_{i}
(\eta_{\alpha}^{-})}W_{i}(\eta_{\alpha}^{-})
\left(\dot{\eta}_{\alpha} - \lambda_{i}(\eta_{\alpha}^{-})\right)
 - \abs{q_{i}(\eta_{\alpha}^{+})}W_{i}(\eta_{\alpha}^{+})
 \left(\dot{\eta}_{\alpha} - \lambda_{i}(\eta_{\alpha}^{+})\right)\Big)
 \nonumber \\
& \quad + \sum_{i = 1}^{3}\abs{q_{i}(b)}
W_{i}(b)\big(\dot{\eta}_{b} + \lambda_{i}(b)\big),     \nonumber
\end{align}
where $\dot{\eta}_{\alpha}$ is the speed of the Hugoniot wave
$\alpha \in \mathcal{J}$, $b = 0^{+}$ stands for the points near the
boundary $\eta=0$, and $\dot{\eta}_{b}$ is the slope of the
boundary, which is actually zero.
For the second equality, we used
the fact that, when $U_0-V_0\in L^1(\mathbb{R}^+)$,
if $\alpha'$ is the uppermost front in $\mathcal{J}$, then
$\lambda_i(\eta_{\alpha'}^+)=0$ for $i=1,2,3$.
Also, suppose that the
front $\beta'\in\mathcal{J}$ is the closest to $\eta=0$ so that
$\lambda_i(b)=\lambda_i(\eta_{\beta'}^-)$. Define
\begin{align}
E_{\alpha,i}& =\abs{q_{i}^{+}}W_{i}^{+}\left(\lambda_{i}^{+}
- \dot{\eta}_{\alpha} \right) - \abs{q_{i}^{-}}W_{i}^{-}
\big(\lambda_{i}^{-} - \dot{\eta}_{\alpha} \big), \\
E_{b,i}& = \abs{q_{i}(b)}W_{i}(b)\big(\dot{\eta}_{b} +
\lambda_{i}(b)\big),
\end{align}
where $q_{i}^{\pm}$ = $q_{i}(\eta^{\pm}_{\alpha})$, $W^{\pm}_{i} =
W_{i}(\eta^{\pm}_{\alpha})$, and $\lambda^{\pm}_{i}$ =
$\lambda_{i}(\eta^{\pm}_{\alpha})$. Then
\begin{equation}
{\dd \over \dd\xi}\mathit{\Phi}\left(U(\xi), V(\xi)\right) =
\sum_{\alpha \in \mathcal{J}}\sum_{i=1}^{3}E_{\alpha,i} +
\sum_{i=1}^{3}E_{b,i}.
\end{equation}

Our main goal is to establish the following bounds:
\begin{eqnarray}
&& \sum_{i=1}^{3}E_{b,i} \leq 0 \quad\text{ (near the boundary)},\label{bdest} \\
&& \sum_{i=1}^{3}E_{\alpha,i} \leq \mathcal{O}(1)\delta\abs{\alpha}
\quad\text{ (when $\alpha$ is a weak wave in
$\mathcal{J}$)}.\label{wkest}
\end{eqnarray}
If these are established,
from \eqref{bdest}--\eqref{wkest}, we conclude
\begin{equation}\label{Phi}
{\dd\over \dd\xi}\mathit{\Phi}\left(U(\xi), V(\xi)\right) \leq
\mathcal{O}(1)\delta.
\end{equation}

If the constant $\kappa_{2}$ in the Lyapunov functional is chosen
large enough, by the Glimm interaction estimates, all the weight
functions $W_{i}(\eta)$ decrease at each ``time'' where two fronts
of $U$ or two fronts of $V$ interact. By the self-similarity
of the Riemann solutions, $\mathit{\Phi}$ decreases at this ``time''.
Integrating \eqref{Phi} over the interval $\left[0, \xi\right]$, we
obtain
\begin{equation}\label{decpot}
\mathit{\Phi}\left(U(\xi), V(\xi)\right) \leq
\mathit{\Phi}\left(U(0), V(0)\right) + \mathcal{O}(1)\delta \xi
\end{equation}
as desired. Actually, \eqref{wkest} has been proved in
\cite{Bressan-Liu-Yang-1999} or \cite{Holden-Risebro-2002} via a case
by case analysis.

We also need to consider what happens to the Lyapunov functional $\Phi$ if a weak wave
reflected off the boundary.
Even in this case, the weights $W_i(\eta)$ can be made to decrease
across ``time" $\xi = \tau$ when a $1$-wave interacts with the boundary.
Using the modified Glimm interaction potential $\mathcal{Q}$
and Lemma \ref{lem:reflection} above,
it holds that, across ``time" $\xi = \tau$ when a weak 3-wave reflected off
the characteristic boundary $\{\eta=0\}$, both $\mathcal{G}$ and $\mathcal{Q}$ decrease,
if $k_{+}$ is chosen sufficiently large (see Section 2).
One sees that the result holds if $\kappa_2 \gg \kappa_1$  in \eqref{3.6}.
What remains to be proved is \eqref{bdest} near the boundary.

\subsection{Estimates near the boundary}
In this section, we focus on the estimate of \eqref{bdest} near the
boundary. It is different from those for the Cauchy problem. We
exploit the exclusive property of the boundary condition in
\eqref{prob}: The flows $U$ and $V$ have the same pressure near
the boundary. Then we construct a piecewise constant weak solution
only along the Hugoniot curves determined by the Riemann date $U(b)$
and $V(b)$, which are the states of $U$ and $V$ near the boundary
(that is, the point $b$ is some fixed point $(\xi_b, \eta_b)$ near the
positive $\xi$-axis with $\eta_b > 0$). Let $h_{i}(b)$ be the
strength of the $i$-th shock in the Riemann problem determined by
$U(b)$ and $V(b)$, and $\lambda_{i}$ be the corresponding wave
speed. Then, as the second-family is linearly degenerate, we infer that
$\lambda_2\equiv 0$. Recalling that $\dot{\eta}_b=0$,
we conclude
\begin{eqnarray}
E_{b,2}\equiv 0.
\end{eqnarray}

To estimate $E_{b,1}+E_{b,3}$, we consider the following two cases.

{\it Case 1: $h_{1}(b) = 0$}. If $h_1(b)=0$,
then there is no first-family wave in the Riemann solution to the Riemann problem
$(U(b),V(b))$.  Since the second-family waves are the characteristic
discontinuities, the pressure of the middle-state $U_m$ keeps
unchanged, i.e. $p_m=\overline{p}$. As the third-family is
genuinely nonlinear, there must be a jump of pressure across a wave
of the third-family. Thus, in this case, there must be no wave of the
third-family, that is, $h_3(b)=0$. Hence, $q_1(b)=0=q_3(b)$. We
conclude in this case that
$$
E_{b,i}=0, \qquad i=1,3.
$$

{\it Case 2: $h_1(b)\ne0$}.  For this case, as analyzed in Case 1,
one concludes that $h_{3}(b) \neq 0$ as well.  Starting from $U(b)$,
go along the 1-Hugoniot curve to reach $U_{1}$, then possibly along
the 2-characteristic Hugoniot curve to reach $U_{2}$, and the 3-Hugoniot
curve to reach $V(b)$. Furthermore, among the first-family and third-family
waves, they must be of distinct type, that is, there can be
a 1-compressive shock (Lax shock) and a 3-decompressive shock (non Lax
shock) or vice versa in the Riemann solution to keep the pressure
unchanged in the approximate solutions $U$ and $V$. We need to show
that
\begin{eqnarray}\label{preest}
|h_3(b)|=\mathcal{O}(1)|h_1(b)|.
\end{eqnarray}
If this is true, observing that $\lambda_1<0$ and $\lambda_3>0$ at
$\overline{U}_+$, so that $q_i(b)=c_i^ah_i(b)$ with the constants
$c_i^a>0$ to be chosen, we have
\begin{eqnarray*}
E_{b,1}+E_{b,3}&=&c_1^a|h_1(b)|W_1(b)\lambda_1(b)+c_3^a|h_3(b)|W_3(b)\lambda_3(b)\\
&=&|h_1(b)|\big(-c_1^aW_1(b)|\lambda_1(b)|+c_3^a\mathcal{O}(1)W_3(b)|\lambda_3(b)|\big)\\
&\le&|h_1(b)|\big(-c_1^a|\lambda_1(b)|+2c_3^a\mathcal{O}(1)|\lambda_3(b)|\big)\\
&\le&0,
\end{eqnarray*}
by using the fact that $\left|\frac{\lambda_3(b)}{\lambda_1(b)}\right|$ is bounded
(depending only on $\overline{U}_+$) and then choosing $c_1^a$
quite large. This completes the proof of \eqref{bdest}.

Then what left is to prove \eqref{preest}. We know
$U_1=S_1(h_1)(U_b)$
and $V_b=S_3(h_3)(U_2)$. For $U$ close to $\overline{U}_+$, since
$\frac{\dd S_1(\alpha)(U)}{\dd \alpha}\left|_{\alpha=0}=r_1(U)\right.$ and
especially the third argument of $r_1$ is $\kappa_1(-\bar{\lambda}_1
u)\ne0$ (it is nonzero at the reference state $\overline{U}_+$), by
the inverse function theorem, we infer that $|h_1(b)|\backsim
|p_1-\overline{p}|$, and similarly for the third-family,
$|h_3(b)|\backsim |p_2-\overline{p}|$.
Since $p_2=p_1$, we conclude
that $|h_1(b)|\backsim|h_3(b)|$. Here, for positive quantities $a$ and $b$,
we use $a\backsim b$ to mean that there is a positive constant
$C_1$ depending only on $\overline{U}_+$ so that $C_1^{-1} b\le
a<C_1 b$.

\section{Existence of a Semigroup}

Using the existence and uniqueness results established
in the earlier sections,
we can now establish the existence of the semigroup $\mathcal{S}$ of solutions
generated by the wave-front tracking algorithm.

\begin{proposition}
Assume that $\TV (\overline{U}(\cdot))$ is sufficiently small.
Then, for $\delta>0$, the map $(\overline{U}(\cdot), \xi) \mapsto
U^{\delta}(\xi,\cdot)\mathrel{\mathop:}
=\mathcal{S}^{\delta}_{\xi}(\overline{U}(\cdot))$ produced by
the wave-front tracking algorithm
is a uniformly Lipschitz continuous semigroup such that

\begin{itemize}
\item[(\rmnum{1})] $\mathcal{S}^{\delta}_{0}\overline{U} = \overline{U},
 \quad \mathcal{S}^{\delta}_{\xi_1}\mathcal{S}^{\delta}_{\xi_2}\overline{U}
 = \mathcal{S}^{\delta}_{\xi_1 + \xi_2}\overline{U}$;
\item[(\rmnum{2})] $\norm{\mathcal{S}^{\delta}_{\xi}\overline{U}
- \mathcal{S}^{\delta}_{\xi}\overline{V}}_{L^1(\mathbb{R}^+)} \leq
C\norm{\overline{U} - \overline{V}}_{L^1(\mathbb{R}^+)} + C\delta
\xi$.
\end{itemize}
\end{proposition}

\begin{proof}
Property (i) follows immediately because $\mathcal{S}^{\delta}$ is
produced by the wave-front tracking method. We now prove property (ii).

Let $U^{\delta}$ and $V^{\delta}$ be the front tracking
$\delta$-approximate solutions of \eqref{prob} with the initial data
$\overline{U}=U_0(\cdot)$ and $\overline{V}=V_0(\cdot)$, respectively. By
\eqref{eqlnorm} and \eqref{decpot}, we obtain that, for any $\xi
\geq 0$,
\begin{eqnarray}
\norm{U^{\delta}(\xi) - V^{\delta}(\xi)}_{L^1(\mathbb{R}^+)} &\leq&
C
\Phi(U^{\delta}(\xi), V^{\delta}(\xi)) \nonumber \\
&\leq& C\Phi(U^{\delta}(0), V^{\delta}(0))
+ C_{1}\mathcal{O}(1)\delta \xi \nonumber \\
&\leq& C_{1}C_{2}\norm{\overline{U} -
\overline{V}}_{L^1(\mathbb{R}^+)} + C_{1}\mathcal{O}(1)\delta \xi.
\nonumber
\end{eqnarray}
This establishes the Lipschitz continuity of the $\delta$-semigroup
with respect to the initial data and time.
\end{proof}

\begin{definition}
Given $\delta_{0} > 0$, define a domain $\mathscr{D}$ as the closure of
the set consisting of the functions $U :\mathbb{R}_+ \mapsto
\mathbb{R}^3$ such that $U$ belongs to $L^1(\mathbb{R}_+; \mathbb{R}^3)$ by modulo a constant and $\BV (U-\overline{U}_+)(\mathbb{R}^+) \leq \varepsilon_{0}$.
\end{definition}

The semigroup $\mathcal{S}$ generated by the wave-front tracking method
is given by the subsequent theorem.

\begin{theorem}\label{semithm}
Let $\TV(\overline{U}(\cdot))$ be sufficiently small.
Then the sequence $\mathcal{S}^{\delta}$ generated from the front tracking algorithm
is a Cauchy sequence in the $L^1$--norm,
and the sequence $\mathcal{S}^{\delta}_{\xi}(\overline{U})$ converges
to a unique limit $\mathcal{S}_{\xi}(\overline{U})$ as $\delta \to 0$.
The map $\mathcal{S} :[0, \infty) \times \mathscr{D} \mapsto \mathscr{D}$
is a uniformly continuous semigroup.
In particular, there exists a constant $L$ such that,
for all $\xi_1, \xi_2 > 0$ and $\overline{U}$, $\overline{V}$ $\in \mathscr{D}$,
\begin{itemize}
\item[(\rmnum{1})] \textbf{Semigroup Property:} $\mathcal{S}_{0}\overline{U} = \overline{U}, \quad \mathcal{S}_{\xi_1}\mathcal{S}_{\xi_2}\overline{U} = \mathcal{S}_{\xi_1 + \xi_2}\overline{U}$;
\item[(\rmnum{2})] \textbf{Lipschitz Continuity:} $\norm{\mathcal{S}_{\xi}\overline{U} - \mathcal{S}_{\xi} \overline{V}}_{L^1(\mathbb{R}^+)} \leq L\norm{\overline{U} -
\overline{V}}_{L^1(\mathbb{R}^+)}$;
\item[(\rmnum{3})] {Each trajectory $\xi \mapsto \mathcal{S}_{\xi}\overline{U}$ yields a weak solution to the initial-boundary value problem \eqref{prob}};
\item[(\rmnum{4})]\textbf{Consistency with Riemann Solver:} {For any piecewise constant initial data $\overline{U} \in \mathscr{D}$, there exists a small $\theta > 0$ such that, for all $\xi \in [0,\theta]$, the trajectory
$U(\xi, \cdot) = \mathcal{S}_{\xi}\overline{U}$ agrees with the
solution of \eqref{prob} obtained by piecing together the standard
entropy solutions of the Riemann problems.}
\end{itemize}
\end{theorem}

Theorem \ref{semithm} implies that a sequence of $\delta$-approximate front tracking solutions to
\eqref{prob} converges to a unique entropy solution whose value is in $\mathscr{D}$,
and this unique limit is $L^1$--stable.
More precisely, we have the following immediate consequence from Theorem \ref{semithm}:

\begin{corollary}\label{co41}
Let $\TV(\overline{U}(\cdot))$ be sufficiently small. Then the entropy
solution to the initial-boundary value problem \eqref{prob} produced by
the wave-front tracking algorithm is $L^1$-stable and unique.
\end{corollary}

The proof of Theorem \ref{semithm} relies on the subsequent crucial estimate
(also used in \cite{Bressan-Colombo-1995}) about the approximation of Lipschitz flows:

Suppose that $\mathcal{S} :[0, \infty) \times \mathscr{D} \to \mathscr{D}$ is
a global semigroup with a Lipschitz constant $L$.
Let $T > 0$, $\overline{Z} \in \mathscr{D}$, and $Z :[0, T] \mapsto \mathscr{D}$
be a continuous mapping taking values in piecewise constant functions,
with jumps along finitely many polygonal lines in the $(\xi,\eta)$--plane.
Then
\begin{equation}\label{Mfest}
\norm{Z(T) - \mathcal{S}_{T}\overline{Z}}_{L^1} \leq L \cdot \left\lbrace \norm{Z(0)
- \overline{Z}}_{L^1} + \int_{0}^{T} \varlimsup_{\mu \rightarrow
0^{+}} \frac{\norm{Z(\xi+\mu)-\mathcal{S}_{\mu}Z(\xi)}_{L^1}}{\mu}\, \dd
\xi\right\rbrace.
\end{equation}

Using the essential estimates shown in the earlier sections, Theorem
\ref{semithm} is shown along a similar line of arguments to
the one followed by Bressan-Colombo in \cite{Bressan-Colombo-1995}.
The only significant difference here between the wave-front tracking
by Bressan (cf. \cite{Bressan-2000}) and the version of the
method (cf. Holden and Risebro \cite{Holden-Risebro-2002}) we employ
is on how to control the number of fronts from growing to infinity
within finite time.
In \cite{Bressan-Colombo-1995}, the accurate Riemann solver ({\bf ARS})
and simplified Riemann solver ({\bf SRS}) are used to construct
the approximate solutions. With ({\bf ARS}), waves of each family can
be possibly introduced in the Riemann solution where every rarefaction wave
is divided into equal parts to obtain a rarefaction fan of wave-fronts,
while with ({\bf SRS}), all new waves are lumped together as a single
non-physical front, traveling faster than all wave speeds.
Furthermore, in \cite{Bressan-Colombo-1995},
the simplified Riemann solver, rather than the accurate Riemann solver,
is employed when the interaction term is less than a cut-off function
in the order of $\sqrt{\delta}$.
In the front tracking method we use, the approximate Riemann solution
is changed by removing weak fronts.
The point of this procedure is that the fronts of high generation are quite weak.
Particularly, in our arguments, when the interaction term is less than $\delta$,
this corresponds to new and removed fronts with generation higher than $N$
(here $N$ is a suitably large positive integer computed using the initial
approximation parameter $\delta$, see \cite{Chen-Kukreja-Yuan-2012}
and \cite{Holden-Risebro-2002} for further details).
We also refer to Chen-Li \cite{Chen-Li-2008} for a related analysis
with the full Euler system where ({\bf SRS}) is used in place of ({\bf ARS})
when the interaction term is less than $\delta$.

Thus, using an argument similar to \cite{Bressan-Colombo-1995},
one concludes that $\mathcal{S}^{\delta_{n}}_{\xi}\overline{Z}_{n}$ is a Cauchy
sequence, converging in the $L^1$--sense as long as the error from
removing weak fronts tends zero as the initial error parameter
$\delta \rightarrow 0$ (this has been proved in \cite{Chen-Kukreja-Yuan-2012}).
Consequently, the map $\mathcal{S} :[0, \infty) \times \mathscr{D} \mapsto \mathscr{D}$
defined as the limit of the approximate solutions produced by the front tracking
algorithm is well-defined.

Now, the proofs of statements (i) to (iv) in Theorem \ref{semithm} are as follows.
Statements (i), (ii), and (iv) are immediate
since $\mathcal{S}$ is the limit of front tracking approximate
solutions $\mathcal{S}^{\delta}$. It is similar to prove (iii) as
\cite{Bressan-Colombo-1995}, but, as discussed above, the wave-front tracking method
we employ here is slightly different. Lastly, we observe that the
entropy solution fulfills the boundary condition of the initial-boundary value
problem (\ref{prob}) due to the construction of our approximate solutions.
This completes the proof of Theorem \ref{semithm}.

\section{Uniqueness in the class of viscosity solutions}

In this section, we analyze an arbitrary Lipschitz semigroup defined
on the domain $\mathscr{D}$ of $\BV$ functions.
In particular, we prove that the semigroup $\mathcal{S}$ is uniquely determined
when the local flow is assigned in connection with the initial data,
a piecewise constant function.
Based on the results in Section 4, we first show that the semigroup $\mathcal{S}$
generated by the wave-front tracking method is the canonical trajectory of
the standard Riemann semigroup (SRS).
Then the uniqueness of entropy solutions is shown to extend to a broader class,
namely the class of viscosity solutions as introduced by Bressan in \cite{Bressan-1995}.
The essential part here is to show that, in the viscosity class,
the entropy solution matches the semigroup trajectory generated by the front tracking method.

\begin{definition}\label{SRS}
Problem \eqref{prob} is said to admit a {\it standard Riemann semigroup}
if, for some small $\varepsilon_{0} > 0$, there is a continuous
map $\mathcal{R} :[0, \infty) \times \mathscr{D} \mapsto \mathscr{D}$ and a
constant $L$ such that for every $\overline{U}, \overline{V} \in \mathscr{D}$ and $\xi_1, \xi_2 \ge 0$ we have:
\begin{itemize}
\item[(\rmnum{1})] $\mathcal{R}_{0}\overline{U} = \overline{U},
 \quad \mathcal{R}_{\xi_1}\mathcal{R}_{\xi_2}\overline{U} = \mathcal{R}_{\xi_1 + \xi_2}\overline{U}$;
\item[(\rmnum{2})] $\norm{\mathcal{R}_{\xi}\overline{U} -
 \mathcal{R}_{\xi}\overline{V}}_{L^1} \leq L\norm{\overline{U} -
  \overline{V}}_{L^1}$;
\item[(\rmnum{3})] If $\overline{U} \in \mathscr{D}$
is piecewise constant, then, for all $\xi \in [0,\varepsilon_0]$ sufficiently small, the trajectory $U(\xi, \cdot) = \mathcal{S}_{\xi}\overline{U}$ coincides with the solution of \eqref{prob} obtained by patching together the standard
entropy solutions of the Riemann problems produced by the discontinuities of  $\overline{U}$.
\end{itemize}
\end{definition}

\begin{theorem}\label{5.1}
Suppose that the initial-boundary value problem \eqref{prob} admits a standard Riemann semigroup $\mathcal{R} :[0,\infty) \times \mathscr{D} \mapsto \mathscr{D}$. Consider $\mathcal{S}$ the semigroup generated by the front tracking algorithm, that is, $\mathcal{S}_{\xi}(\overline{U}) = \lim_{\delta \rightarrow 0}S^{\delta}_{\xi}(\overline{U})$ with $\overline{U} \in \mathscr{D}$. Then, for all $\xi \ge 0$ and $\overline{U} \in \mathscr{D}$, $\mathcal{R}_{\xi}\overline{U} = \mathcal{S}_{\xi}\overline{U}$.
\end{theorem}

Theorem \ref{5.1} is proved using similar arguments as in \cite{Bressan-1995}
based on the fundamental estimate \eqref{Mfest}
and the essential feature of the local flow in the $\xi$-direction,
that is, the wave-front tracking method and the standard Riemann semigroup
both have the structure of the Riemann solutions.

In what follows, we discuss some necessary and sufficient conditions
for a function $\xi \mapsto U(\xi) \in \mathscr{D}$ to coincide with
a semigroup trajectory.
Following Bressan \cite{Bressan-1995},
there are two types of local approximate parametrices for system \eqref{Lagr}.

The first approximate parametrice comes from the self-similar solution of
the Riemann problem.
To that end, consider a function $U :[0, \infty) \times \mathbb{R}_+ \mapsto \mathbb{R}^{3}$
and a fixed point $(\tau, \zeta)$ in the domain of $U$.
Suppose that $U(\tau, \cdot) \in \mathscr{D}$.
The boundedness of the total variation implies that the following limits exist:
$$
U^- = \lim_{\eta \to \zeta^{-}}U(\tau, \eta), \hspace{8mm} U^+ =
 \lim_{\eta \to \zeta^{+}}U(\tau, \eta).
$$
Let $\vartheta = \vartheta(\xi, \eta)$ be the solution of the Riemann problem with piecewise constant data $U^-$ and $U^+$
and, for $\xi > \tau$, define the following function:
\begin{equation}
H^{\sharp}_{(U, \tau, \zeta)}(\xi, \eta) =
\begin{cases} \vartheta(\xi - \tau, \eta - \zeta) & \text{if } |\eta - \zeta |
\leq \hat{\lambda}(\xi - \tau), \\
U(\tau, \eta) & \text{if } |\eta - \zeta | > \hat{\lambda}(\xi - \tau).
\end{cases} \nonumber
\end{equation}
Here $\hat{\lambda}$ is an upper bound for all wave speeds,
\begin{equation}\label{maxsp}
\sup_{U}|\lambda_{k}(U)| < \hat{\lambda}, \qquad k = 1,2,3.
\end{equation}

The other required parametrice is obtained from the corresponding quasilinear hyperbolic system \eqref{eq:3.15}
$$
A(U)\partial_\xi U + B(U)\partial_\eta U= 0
$$
by ``freezing" the coefficients of
the matrices ${A}(U)$ and ${B}(U)$ in a neighbourhood of the state $U(\tau, \zeta)$.

For $\xi > \tau$, define $H^{\flat}_{(U, \tau, \zeta)}$ as the solution of the
linear Cauchy problem with constant coefficients:
\begin{equation}\label{LCauchy}
\tilde{A}\partial_\xi M + \tilde{B}\partial_\eta M = 0, \hspace{8mm} M(\tau, \eta)
= U(\tau, \eta).
\end{equation}
where $\tilde{A}\mathrel{\mathop:}= {A}(U(\tau, \zeta))$ and
$\tilde{B}\mathrel{\mathop:}= {B}(U(\tau, \zeta))$ the matrices evaluated at the fixed state $U(\tau, \zeta)$.
Note that the functions $H^{\#}$ and $H^{\flat}$ depend on the values
$U(\tau, \zeta)$ and $U(\tau, \zeta\pm)$.
In the remaining, we introduce the class of viscosity solutions
and discuss how these solutions share the same local characterization as $H^{\#}$ and $H^{\flat}$.

\begin{definition}
Assume that $U :[0, T] \mapsto \mathscr{D}$ is a continuous map with respect to the
$L^1$--norm.
Then the function $U$ is called a viscosity solution of problem \eqref{prob}
if there are constants ${C}$ and $\hat{\lambda}$
providing bound \eqref{maxsp} such that, with $\beta$ and $\delta$ small enough,
we have, for every $(\tau, \zeta) \in [0, T) \times \mathbb{R}_+$,
\begin{eqnarray}
\frac{1}{\delta}\int\limits_{\zeta - \beta +
\delta\hat{\lambda}}^{\zeta + \beta - \delta\hat{\lambda}} |U(\tau +
\delta, \eta) - H^{\#}_{(U, \tau, \zeta)}(\xi, \eta)| \, d\xi
&\leq& {C}\TV\lbrace U(\tau): (\zeta - \beta, \zeta)
\cup (\zeta, \zeta + \beta) \rbrace, \nonumber \\
\frac{1}{\delta}\int\limits_{\zeta - \beta +
\delta\hat{\lambda}}^{\zeta + \beta - \delta\hat{\lambda}}| U(\tau +
\delta, \eta) - H^{\flat}_{(U, \tau, \zeta)}(\xi, \eta)| \, d\xi &\leq&
{C} \big(\TV\lbrace U(\tau): (\zeta - \beta, \zeta + \beta)
\rbrace\big)^{2}.  \nonumber
\end{eqnarray}
\end{definition}

\begin{theorem}\label{5.2}
Suppose that problem \eqref{prob} admits a standard Riemann semigroup $\mathcal{R}$.
Then a continuous mapping $U :[0, T] \mapsto \mathscr{D}$ is said to be a viscosity
solution of \eqref{prob} if and only if
\begin{equation}
U(\xi, \cdot) = \mathcal{R}_{\xi}\overline{U} \hspace{3mm} \text{at every }
\xi\in [0, T].
\end{equation}
\end{theorem}

Using Theorem 5.2, one concludes that, for problem \eqref{prob},
the entropy solution is unique in the viscosity class of solutions,
which coincides with the semigroup trajectory $\mathcal{S}_{\xi}\overline{U}$
generated by the front tracking method. More precisely, we have

\begin{corollary}
A continuous mapping $U :[0, T] \mapsto
\mathscr{D}$ is a viscosity solution if and only if
\begin{equation}
U(\xi, \cdot) = \mathcal{S}_{\xi}\overline{U} \hspace{3mm} \textit{for any }
\xi\in [0, T].
\end{equation}
\end{corollary}

With the estimates in Sections 2-3,
the proof here follows along a similar line of reasoning to
the one presented in \cite{Bressan-1995}.
The only difference is a straight boundary and physical domain
restricted to the positive $\eta$-axis;
nonetheless, we can still proceed with the proof provided that
the convergence of the wave-front
tracking method is obtained which has been outlined in Section 2 and
carried out in Eulerian coordinates in \cite{Chen-Kukreja-Yuan-2012}.

\begin{remark}
Note that, in simpler cases such as the potential flow,
isentropic or isothermal Euler flow, as far as the $L^1$--stability
problem is concerned,
we achieve the same results as the full Euler equations.
\end{remark}

\section{Uniqueness of Solutions to the Free Boundary
Problem in  Eulerian Coordinates}

In this section we apply the following Wagner's Theorem
\cite[Theorem 2]{Wagner-1987} to show
the uniqueness of entropy solution to
the free boundary problem \eqref{probeuler}.

\begin{theorem}[Wagner \cite{Wagner-1987}]\label{thmwagner}
Consider
\begin{eqnarray}\label{eqwagner}
\partial_t U+ \partial_x F(U)=0,
\end{eqnarray}
where $(t,x)\in\R^2$, $U(t,x)=(u_1,\cdots, u_n)^\top$,
and $F(U)=(f_1(U),\cdots, f_n(U))^\top\in\mathbb{R}^n$
is a smooth vector-function.
For any bounded measurable
solutions of \eqref{eqwagner}, with $u_1(t,x)\ge0$, let $(t,y)$
satisfy
$$
\frac{\p y}{\p x}=u_1(t,x),\qquad \frac{\p y}{\p t}=-f_1(U(t,x))
$$
in the sense of distributions. Then $T: (t,x)\mapsto (t, y(t,x))$ is
a Lipschitz-continuous transformation, which induces a one-to-one
correspondence between $L^\infty$ weak solutions of \eqref{eqwagner}
on $\mathbb{R}^+\times\mathbb{R}$ satisfying $0<\varepsilon\le
u_1(t,x)\le M<\infty$ for some $\varepsilon$ and $M$, and $L^\infty$
weak solutions of
\begin{eqnarray}
&&\partial_t\big(\frac{1}{u_1}\big)-\partial_y\big(\frac{f_1(U)}{u_1}\big)=0,\nonumber\\
&&\partial_t\big(\frac{1}{u_1}(u_2,\cdots,u_n)^\top\big)+\partial_y\big((f_2(U),\cdots,
f_n(U))^\top-\frac{1}{u_1}f_1(U)(u_2,\cdots,u_n)^\top\big)=0, \label{eqwag2}
\end{eqnarray}
on $\mathbb{R}^+\times\mathbb{R}$ satisfying $\varepsilon\le
u_1(t,x)\le M$.

If $\eta(U)$ is any convex extension of
\eqref{eqwagner}, i.e., there is a flux $q(U)$ such that $\nabla\eta
\nabla F=\nabla q,$ so that $\partial_t\eta+\partial_x q=0$ for classical solutions, then any
solution of \eqref{eqwagner}, satisfying
\begin{eqnarray}\label{eqwag3}
\partial_t\eta(U)+ \partial_x q(U)\le0
\end{eqnarray}
in the sense of distributions in $(t,x)$
corresponds to a solution of \eqref{eqwag2}, satisfying
\begin{eqnarray}\label{eqwag4}
\partial_t \tilde{\eta}(V)+\partial_y \tilde{q}(V)\le 0,
\end{eqnarray}
in the sense of distributions in $(t,y)$,
where $V=(v_1,\cdots, v_n)^\top$, $\tilde{\eta}(V)=\frac{\eta(U)}{u_1}$, and
$\tilde{q}(V)=q(U)-f_1(U)\tilde{\eta}(V)$. Furthermore, $\eta$ is
convex if and only if $\tilde{\eta}$ is convex as a function of $V$.
Thus, the Lax's entropy inequality holds for a solution of
\eqref{eqwagner} if and only if it holds for the corresponding
solution of \eqref{eqwag2}.
\end{theorem}
As indicated by Wagner \cite[p.123]{Wagner-1987}, this theorem may
also be applied to initial-boundary value problems by slight
modification. Thus, we can also use it in our case. Suppose that $U_0^1(y)$
and $U_0^2(y)$ are two initial data functions of problem \eqref{probeuler} in
Eulerian coordinates, and $(U^1, g^1)$ and $(U^2, g^2)$ are the
corresponding weak entropy solutions obtained by the front tracking
method. We have known that $U^1$ and $U^2$ must be bounded. We now show
that $U^1=U^2$ and $g^1=g^2$, provided $U_0^1=U_0^2$.

Suppose that $V^1$ and
$V^2$ are the solutions of \eqref{prob} in Lagrangian coordinates
that correspond to $U^1$ and $U^2$, respectively. By Theorem \ref{thmwagner},
it suffices to show that $V^1=V^2$. Then, from \eqref{eqfree}, we may find
$g^1=g^2$. Therefore, by the uniqueness results of problem \eqref{prob}
(Corollary \ref{co41}), it suffices to prove the following lemma.
\begin{lemma}
If $U_0^1(y)=U_0^2(y)$, then $V^1(0,\eta)=V^2(0,\eta)$.
\end{lemma}

\begin{proof}
For $j=1,2$, let $T^j: (x,y)\mapsto (\xi,\eta)$ be the Lagrangian
transform \eqref{eq:3.10} associated with the solution $U^j$. They
are Lipschitz-continuous and one-to-one. Furthermore, we know
$U^j(x,\cdot)$ is of bounded variation and $\lim_{x\to
x_0}\norm{U^j(x,\cdot)-U^j(x_0,\cdot)}_{L^1}=0$.
Thus, we can not only
solve $\eta$ from \eqref{etaeq}, but also those equations make sense
on each line $x=\mathrm{constant}$.

We first show that, if $U_0^1(y)=U_0^2(y)$, then $T^1|_{x=0}=T^2|_{x=0}$.

Suppose $(\xi_j,\eta_j)=T^j(x,y)$. Recall by definition that $\xi=x$.
Then, at $x=0$, we have both $\xi_1=\xi_2=0$.
Note that $\eta_j$ satisfies
$\eta_j(0,0)=0$ and solves
$$
\frac{\p \eta_j(x,y)}{\p y}|_{x=0}=(\rho u)^{j}(0,y).
$$
Since $(\rho u)^{1}(0,y)=(\rho u)^{2}(0,y)$ from
$U_0^1(y)=U_0^2(y)$, we conclude that $\eta_1(0,y)=\eta_2(0,y)$ as
well. This shows that $T^1(0,y)=T^2(0,y)$, which implies that
$(T^1)^{-1}(0,\eta)=(T^2)^{-1}(0,\eta)$ and lies on $\{x=0\}$.
Since
$$
V^j(\xi,\eta)=U^j((T^j)^{-1}(\xi,\eta)),
$$
we conclude
$$
V^1(0,\eta)=U^1((T^1)^{-1}(0,\eta))=U^1((T^2)^{-1}(0,\eta))=
U^2((T^2)^{-1}(0,\eta))=V^2(0,\eta).
$$
\end{proof}

Finally, summarizing all the analysis above, we state the main
theorem of this paper.

\begin{theorem}\label{thmmain}
For the initial data $U_0$ close to the reference state
$\overline{U}_+$ in the sense that
$\norm{U_0-\overline{U}_+}_{\BV}$ is sufficiently small, there is
one and only one weak entropy solution  $(U,g)$ to problem
\eqref{probeuler} constructed by the front tracking method. Furthermore,
reformulating this problem in Lagrangian coordinates as problem
\eqref{prob}, then the $L^1$--stability holds in the sense that
$$\norm{V^1(\xi)-V^2(\xi)}_{L^1(\mathbb{R}^+)}\le
C \norm{V^1(0)-V^2(0)}_{L^1(\mathbb{R}^+)}$$ for any two solutions
$V^1$ and $V^2$ of \eqref{prob}. The solution is also unique in the
class of viscosity solutions in Lagrangian coordinates.
\end{theorem}

\bigskip

{\bf Acknowledgments.} The research of Gui-Qiang Chen was supported in part by the National Science Foundation under Grant DMS-0807551, the UK EPSRC Science and Innovation Award to the Oxford Centre for Nonlinear PDE (EP/E035027/1), the NSFC under a joint project Grant 10728101, and the Royal Society-Wolfson Research Merit Award (UK).
Vaibhav Kukreja was supported in part by the National Science Foundation under Grant DMS-0807551 and
the UK EPSRC Science and Innovation Award to the Oxford Centre for Nonlinear PDE (EP/E035027/1).
Hairong Yuan is supported in part by National Natural Science Foundation of China under Grant No. 10901052,
China Scholarship Council (No. 2010831365),
Chenguang Program (No. 09CG20) sponsored by Shanghai Municipal Education Commission
and Shanghai Educational Development Foundation,
and a Fundamental Research Funds for the Central Universities.


\end{document}